\documentclass{article}
\usepackage{graphicx} 
\usepackage{latexsym}
\usepackage{amsfonts, mathrsfs, amssymb, amsmath}
\usepackage{a4wide,fullpage}
\usepackage{url}
\usepackage{enumerate}
\usepackage{bm}
\usepackage{amsthm}
\usepackage{color}
\usepackage{bbm}
\usepackage{natbib}
\usepackage[colorlinks,citecolor=blue,urlcolor=blue]{hyperref}

\newcommand\E{{\mathbb E}}

\newcommand\given{\, \vert \, }

\newtheorem{theorem}{Theorem}

\newtheorem{lemma}{Lemma}

\newtheorem{proposition}{Proposition}
\newtheorem{remark}{Remark}

\allowdisplaybreaks

\title{Limit Laws for the Distance to Fr\'echet Means of Random Graphs}
\author{Qunqiang Feng, Zixin Tang, and Zhishui Hu{\thanks{Email: huzs@ustc.edu.cn}}\\
{\small Department of Statistics and Finance, School of Management} \\
{\small University of Science and Technology of China}\\
{\small Hefei 230026, China}}
\date{}

\begin{document}
\maketitle

\begin{abstract}
This paper investigates the Fr\'echet mean of the Erd\H{o}s-R\'enyi random graph $G_{n,p}$ 
with respect to the Frobenius distance on graph Laplacians, 
a metric that captures global structural information beyond local edge flips.
We first characterize the Fr\'echet mean set as consisting of quasi-regular graphs 
(i.e., graphs where all vertex degrees differ by at most one). 
We then analyze the asymptotic behavior of the Frobenius distance $F_n=d_{\mathrm{F}}(G_{n,p},R)$ as $n\to\infty$, 
where $R$ is any Fr\'echet mean. 
Closed-form expressions for the mean and variance of $F_n^2$ are derived, which are invariant to the choice of $R$.
Leveraging these results, 
we establish several weak convergence laws for the Frobenius distance over all regimes of $p \in (0,1)$
as $n \to \infty$. 
Finally, under the scaling condition $n^2 p(1-p) \to \infty$ we prove the asymptotic normality of this distance, 
which exhibits a phase transition governed by the growth rate of $np(1-p)$. 
Our results reveal how metric selection fundamentally shapes Fr\'echet mean geometry in random graphs.

\bigskip
\noindent{\it Keywords:} Erd\H{o}s-R\'{e}nyi random graph; 
Frobenius metric; regular graph; dependency graph; Stein's method

\noindent{\it AMS 2020 Subject Classification}: Primary
   05C80,    
   60C05;     
   secondary
   60F05   
\end{abstract}

\section{Introduction}

The {\em Fr\'echet mean}, introduced by Maurice Fr\'echet \cite{Frechet1948les}, 
generalizes the concept of expectation from Euclidean space to general metric spaces $(\mathcal{X},d)$.
It is defined as any minimizer of the {\em Fr\'echet function}
\begin{equation}\label{frefun}
   f_{\mu}(x)=\int_{\mathcal{X}} d^2(x,y){\rm d}\mu(y), \quad x\in \mathcal{X}, 
\end{equation}
where $\mu$ is a probability measure on $\mathcal{X}$. 
The existence and uniqueness of the Fr\'echet mean depend on the topological properties of the underlying metric space
\cite{karcher1977riemannian,hotz2015intrinsic,Cao2025}.   
When uniqueness fails, 
all the minimizers usually form a closed set \cite{Evans2024limit}, termed the {\em Fr\'echet mean set}.

Owing to its applicability to non-Euclidean data, the Fr\'echet mean has gained significant traction 
in probability theory \cite{Barden2013,Schotz02022strong,Evans2024limit}, 
statistics \cite{Bhattacharya2012,Dai2018,Petersen2019,McCormack2022,McCormack2023,Aveni2025,Torres-Signes2025}, 
and machine learning \cite{zhou2022network,Bouchard2024}. It has found particular applications in network data analysis 
\cite{Ginestet2017,Kolaczyk2020,Lunagomez02102021}, 
where graph-valued data inherently resides in non-Euclidean spaces.

Meyer \cite{meyer2021frechet} recently derived the Fr\'echet mean for inhomogeneous 
Erd\H{o}s-R\'enyi random graphs \cite{bollobas2007phase} using the Hamming distance 
(defined via adjacency matrices; see the next section). 
For the homogeneous Erd\H{o}s-R\'enyi random graph $G_{n,p}$ on $n$ vertices with edge probability $p$, 
an immediate consequence of Meyer \cite{meyer2021frechet} is that the Fr\'echet mean of $G_{n,p}$ is 
the empty graph when $p \le 1/2$ and the complete graph otherwise. 
However, these extreme graphs lie on the ``boundary" of the graph space and fail to capture 
typical topological features of $G_{n,p}$ for general $p\in (0,1)$. 
While the Hamming distance measures local edge flips and is sensitive to sparsity \cite{donnat2018}, 
the Frobenius distance defined through graph Laplacians
incorporates global structural information beyond local connectivity \cite{Chung1997}.
As this distance is increasingly used in network analysis \cite{zhou2022network,brouwer2012}, 
in this paper we propose studying the Fr\'echet mean of Erd\H{o}s-R\'enyi random graphs under the Frobenius metric.

In contrast to the results in Meyer \cite{meyer2021frechet}, we establish that under the Frobenius distance,
the Fr\'echet mean of $G_{n,p}$ is the empty graph only if $np < 1$, and the complete graph only if $n(1-p) < 1$.
Furthermore, for any positive integer $n$ and $p\in (0,1)$, the Fr\'echet mean set of $G_{n,p}$ comprises quasi-regular graphs 
(see remarks following Theorem \ref{theo:frechet}). 
Although uniqueness does not always hold, 
we additionally characterize the asymptotic behavior 
of the Frobenius distance between $G_{n,p}$ and elements of its Fr\'echet mean set, 
as the graph size $n$ tends to infinity.  

To this end, we first derive an explicit expression for the Fr\'echet function of $G_{n,p}$ in terms of vertex degrees, 
which reduces the problem to minimizing a separable quadratic form. 
By analyzing the resulting degree sequences, we obtain a complete description of the Fr\'echet mean set. 
The characterization reveals a rich structure: when $np$ is an integer, 
any $(np-1)$-regular or $np$-regular graph can be a Fr\'echet mean; otherwise, 
the Fr\'echet mean set consists of nearly regular graphs whose degrees differ by at most one.

The main technical contribution of the paper is a detailed analysis of the asymptotic distribution of the Frobenius distance
$F_n$ between $G_{n,p}$ and any Fr\'echet mean graph of it.
The limiting behavior exhibits a phase transition governed by the growth rate of $np(1-p)$. 
When $np(1-p)\to 0$ but $n^2p(1-p)\to\infty$ (very sparse regime) or $np(1-p)\to c>0$ (sparse regime), 
the squared distance $F_n^2$ converges to a normal distribution after appropriate centering and scaling; 
we prove this using a dependency graph approach. In the dense regime where $np(1-p)\to\infty$, 
the dependency graph becomes too dense for a direct application, 
and we instead employ a refined normal approximation method based on discrete-difference Stein bounds 
recently developed by Shao and Zhang \cite{shao2025104574}. 
This allows us to establish asymptotic normality for $F_n^2$ with an explicit variance formula. 
Applying the delta method then yields the corresponding limit laws for $F_n$ itself. 
Our results show that the Frobenius distance behaves differently from the Hamming distance, 
capturing the intrinsic global structure of Erd\H{o}s-R\'enyi random graphs.

Throughout this paper, we shall use the following notation.
For any real number $x$, let $\lfloor x\rfloor$ denote its floor function 
(i.e., the greatest integer less than or equal to $x$), 
and let $\{x\}=x-\lfloor x\rfloor$ represent its fractional part.
For two sequences of positive numbers $a_n$ and $b_n$,
we write $a_n\sim b_n$ if $a_n/b_n\to 1$ as $n\to\infty$. 
We write $a_n=\Theta(b_n)$  if there exist constants $0<c_1\le c_2<\infty$ such that
$c_1\le \liminf_{n\to \infty} \frac{a_n}{b_n}\le \limsup_{n\to \infty} \frac{a_n}{b_n}\le c_2$.
We denote by $\operatorname{diag}(x_1,x_2,\dots,x_n)$ the diagonal matrix 
whose $i$-th diagonal entry is $x_i$ for each $i=1,2,\dots,n$.
For an $n\times n$ matrix $A=(a_{ij})$, the Frobenius norm of $A$ is defined as $\|A\|_{\text{F}}=(\sum_{i=1}^n\sum_{j=1}^na_{ij}^2)^{1/2}$.
For probabilistic convergence, let
$\xrightarrow{D}$ and $\xrightarrow{P}$ denote convergence in distribution and convergence in probability, 
respectively.
 
The rest of the paper is organized as follows. 
Section 2 reformulates the Fréchet function for $G_{n,p}$ via vertex degrees and characterizes its Fr\'echet mean set 
as consisting of quasi-regular graphs. 
Section 3 derives the mean and variance of the squared Frobenius distance between $G_{n,p}$ and its Fr\'echet means, 
and establishes weak convergence results with a phase transition driven by $n^2p(1-p)$. 
Section 4 proves the asymptotic normality of this distance under the condition $n^2p(1-p)\to\infty$: 
a dependency graph method is used for the very sparse and sparse regimes, 
while a refined Stein bounds-based normal approximation is adopted for the dense regime; 
the delta method then yields the corresponding asymptotic normality for the Frobenius distance itself.

\section{Fr\'echet mean set for Erd\H{o}s-R\'enyi random graphs}

We begin by defining a metric space on the set $\mathcal{G}_n$ of all simple graphs with vertex set
$[n]:=\{1,2,\dots,n\}$, where $n\ge2$ is a fixed integer. 
Clearly, the space $\mathcal{G}_n$ has cardinality $|\mathcal{G}_n|=2^{\binom{n}{2}}$.
Here we consider two fundamental matrices for graphs: the adjacency matrix and the Laplacian matrix \cite{Chung1997}.   
For a graph $G \in \mathcal{G}_n$, its \emph{adjacency matrix} $A = (a_{ij})$ is the $n \times n$ matrix defined by
\[
a_{ij}=\left\{\begin{array}{cl}
     1, &  \mbox{if there is an edge between vertices $i$ and $j$ in $G$};\\
     0, &  \mbox{otherwise}, 
\end{array}\right. 
\]
with $a_{ii}=0$ for all $i\in[n]$.
The {\em  Laplacian matrix} $L=(l_{ij})$ of $G$ is given by
\begin{align*}
  L=D - A,  
\end{align*}
where $D = \operatorname{diag}(D_1, D_2, \dots, D_n)$ is the degree matrix with $D_i = \sum_{j=1}^n a_{ij}$ 
recording the degree of vertex $i$. 
From the basic knowledge of graph theory, 
there exist bijective correspondences between $G$, $A$, and $L$. 
Consequently, we can equip $\mathcal{G}_n$ with matrix-induced distances. 
The Frobenius distance is arguably the simplest of these metrics, 
and it is a common choice for the metrics on the space of graph Laplacians (see, e.g., \cite{zhou2022network}).

The {\em Frobenius distance} between graphs $G,G' \in \mathcal{G}_n$ 
can be defined using adjacency matrices,
\begin{align}\label{dFGGA}
    d_{\text{F}}^A(G,G')=\|A-A'\|_{\text{F}}=\left(\sum_{i,j\in[n]}\big(a_{ij}-a_{ij}'\big)^2\right)^{\frac12},
\end{align}
or Laplacian matrices,
\begin{align}\label{dFGGL}
    d_{\text{F}}^L(G,G')=\|L-L'\|_{\text{F}}=\left(\sum_{i,j\in[n]}\big(l_{ij}-l_{ij}'\big)^2\right)^{\frac12},
\end{align}
where $A'$ and $L'$ correspond to $G'$.

Since all entries in the adjacency matrix are equal to 0 or 1, 
we have $|a_{ij}-a_{ij}'| = (a_{ij}-a_{ij}')^2$ for every pair $(i,j)$. 
Hence, the Hamming distance between $G$ and $G'$ satisfies 
\[
d_{\text{H}}(G,G')=\frac12\sum_{i,j\in[n]}|a_{ij}-a_{ij}'|=\frac12\big[d_{\text{F}}^A(G,G')\big]^2,
\]
by \eqref{dFGGA}.
It thus follows that $d_{\text{F}}^A$ and $d_{\text{H}}$ are equivalent metrics on $\mathcal{G}_n$.
As described in the introduction, 
we will work exclusively with the Laplacian-based distance \eqref{dFGGL},
and use the notation $d_{\text{F}}$ for brevity (omitting the superscript $L$).

For any $G\in \mathcal{G}_n$, let $\overline{G}\in \mathcal{G}_n$ denote its complement graph, i.e.,
any two distinct vertices are adjacent in $\overline{G}$ if and only if they are not adjacent in $G$.
Since the Frobenius norm is invariant under sign change, one can obtain the symmetry
\begin{equation}\label{basicpdF}
d_{\text{F}}(G,G')= d_{\text{F}}(\overline{G},\overline{G'}).
\end{equation}

Define the space of adjacency matrices as
\[
\mathcal{A}_n=\big\{A\in\{0,1\}^{n\times n}: a_{ii}=0,\ a_{ij}=a_{ji}\ \text{for all }i,j\in[n]\big\}.
\]
Let $G_{n,p}$ be the Erd\H{o}s-R\'enyi random graph on $[n]$ with edge probability $p\in(0,1)$. 
Then, for any fixed graph $G\in\mathcal{G}_n$ with adjacency matrix $A=(a_{ij})\in\mathcal{A}_n$, 
the probability mass function of $G_{n,p}$ is
\[
\mathbb{P}(G_{n,p}=G)=\prod_{1\le i<j\le n}p^{a_{ij}}(1-p)^{1-a_{ij}}.
\]
Within the Fr\'echet mean framework (see \eqref{frefun} for the general definition), 
the Fr\'echet function of $G_{n,p}$ is given by the expected squared distance to a candidate graph $G$:
\begin{equation}\label{FrechetfG}
f(G)=\E\big[d_{\text{F}}^2(G,G_{n,p})\big]=\sum_{G'\in\mathcal{G}_n} d_{\text{F}}^2(G,G')\,\mathbb{P}(G_{n,p}=G'), 
\qquad G\in\mathcal{G}_n.
\end{equation}
This function admits the following explicit representation.

\begin{proposition}\label{Prop:Ffunction}
For any graph $G \in \mathcal{G}_n$ with vertex degrees $D_1, D_2, \dots, D_n$, we have that
 \begin{equation}\label{FrefunfA}
    f(G)=c+\sum_{i=1}^n \big[D_i^2-(2np-1)D_i\big],   
\end{equation}
where the constant $c=2n(n-1)p+n(n-1)(n-2)p^2$ is independent of the particular graph $G$.
\end{proposition}

\begin{proof}
Let $A=(a_{ij})$ and $D = \operatorname{diag}(D_1, D_2, \dots, D_n)$ denote the adjacency matrix 
and the degree matrix of the fixed graph $G$, respectively.  
For the random graph $G_{n,p}$, introduce the edge indicators $I_{ij}$ for $1\le i<j\le n$, 
with $I_{ji}=I_{ij}$ and $I_{ii}=0$ for all $i$.  
Thus $A_{n,p}=(I_{ij})$ is its adjacency matrix, 
and its degree matrix is $D_{n,p}=\operatorname{diag}(X_1,X_2,\dots,X_n)$ 
where $X_i=\sum_{j\neq i}I_{ij}$ is the degree of vertex $i$ in $G_{n,p}$.

By definition of the Frobenius distance on Laplacian matrices (\ref{dFGGL}), we have
\begin{align}\label{dF2GGnp}
d_{\text{F}}^2(G,G_{n,p})&= \| (D - A) - (D_{n,p} - A_{n,p}) \|_{\text{F}}^2  \notag\\
&= \| (D - D_{n,p}) - (A - A_{n,p}) \|_{\text{F}}^2  \notag\\
&=  \| D - D_{n,p} \|_{\text{F}}^2  +  \| A - A_{n,p} \|_{\text{F}}^2,
\end{align}
where the cross‑term vanishes since $D-D_{n,p}$ is diagonal while $A-A_{n,p}$ has zero diagonal; 
their Frobenius inner product is therefore zero.

We now evaluate the two remaining terms separately.
For the diagonal term, we have
\begin{align*}
\| D - D_{n,p} \|_{\text{F}}^2 = \sum_{i=1}^n ( D_i - X_i)^2= \sum_{i=1}^n \big( D_i^2 -2D_i X_i+X_i^2\big).
\end{align*}
For each $i$, expanding $X_i^2$ gives
\[
X_i^2= \bigg( \sum_{j \neq i} I_{ij} \bigg)^2 = \sum_{j \neq i} I_{ij} +  \sum_{j\neq i}\sum_{k\neq i,j} I_{ij}I_{ik}
=X_i+  \sum_{j\neq i}\sum_{k\neq i,j} I_{ij}I_{ik},
\]
where we used $I_{ij}^2=I_{ij}$.
Hence, we can proceed with
\begin{align}\label{DDnpF2}
\| D - D_{n,p} \|_{\text{F}}^2 =\sum_{i=1}^n \bigg(D_i^2 -(2D_i-1)X_i+ 
                       \sum_{j\neq i}\sum_{k\neq i,j}I_{ij}I_{ik}\bigg).
\end{align}
Since $a_{ij}^2=a_{ij}$ and $I_{ij}^2=I_{ij}$, for the off-diagonal term we have 
\begin{align}\label{AAnpF2}
\| A - A_{n,p} \|_{\text{F}}^2 &= \sum_{i=1}^n \sum_{j\neq i} (a_{ij} - I_{ij})^2\notag\\
                      &= \sum_{i=1}^n \sum_{j\neq i}\big(a_{ij}-2a_{ij}I_{ij}+I_{ij}\big)\notag\\
                      &= \sum_{i=1}^n\Big(D_i+X_i-2\sum_{j\neq i}a_{ij}I_{ij}\Big).
\end{align}
Substituting \eqref{DDnpF2} and \eqref{AAnpF2} into \eqref{dF2GGnp} yields that
\begin{align}\label{dF2GGnp2}
  d_{\text{F}}^2(G,G_{n,p})=\sum_{i=1}^n \bigg(D_i^2 +D_i-2(D_i-1)X_i-2\sum_{j\neq i}a_{ij}I_{ij}+ 
                       \sum_{j\neq i}\sum_{k\neq i,j}I_{ij}I_{ik}\bigg).  
\end{align}

Note that for any distinct vertices $i,j,k\in[n]$,
\[\E[X_i] = (n-1)p,\quad \E[I_{ij}] = p, \quad \E[I_{ij}I_{ik}] = p^2.\]
Taking expectations on both sides of \eqref{dF2GGnp2}, by \eqref{FrechetfG} we have
\begin{align*}
    f(G)&=\sum_{i=1}^n \big[D_i^2 +D_i-2(D_i-1)(n-1)p-2pD_i+(n-1)(n-2)p^2\big]\notag\\
        &=2n(n-1)p+n(n-1)(n-2)p^2+\sum_{i=1}^n \big[D_i^2-(2np-1)D_i\big],
\end{align*}
which completes the proof of Proposition \ref{Prop:Ffunction}.
\end{proof}

Proposition \ref{Prop:Ffunction} shows that the Fr\'echet function for Erd\H{o}s-R\'enyi random graphs can be decomposed
into a sum of quadratic functions of the vertex degrees. 
Consequently, characterizing the Fr\'echet mean set of $G_{n,p}$  reduces to minimizing $f(G)$
over possible graphical degree sequences $(D_1,D_2,\dots,D_n)$
(i.e., sequences of non‑negative integers with even sum that are realizable by some simple graph on $n$ vertices 
\cite{Tripathi2003}).

Minimizing $f(G)$ in \eqref{FrefunfA} can be viewed as an integer optimization problem with a separable objective.  
Ignoring graphicality constraints, each term $D_i^2-(2np-1)D_i$ is a convex quadratic in $D_i$; its unconstrained minimum is attained at $D_i = np-\tfrac12$.  
Since $D_i$ must be an integer between $0$ and $n-1$, the individually optimal choice is the integer closest to $np-\tfrac12$.  
A simple calculation shows that
\begin{itemize}
    \item If $np$ is an integer, the closest integers are $np$ and $np-1$.
    \item If $np$ is not an integer, the closest integer is $\lfloor np\rfloor$.
\end{itemize}  
Taking these constraints into account leads to the following characterization of the Fr\'echet mean set.

\begin{theorem}\label{theo:frechet}
A graph $G \in \mathcal{G}_n$ belongs to the Fr\'echet mean set of $G_{n,p}$ 
if and only if  its degree sequence $(D_1,D_2,\dots,D_n)$ satisfies
\begin{itemize}
    \item[(i)]  When $np$ is an integer, then $\sum_{i=1}^n D_i$ is even and $D_i \in \{np, np-1\}$ for each $i \in [n]$;
    \item[(ii)] When $np$ is not an integer and $n\lfloor np\rfloor$ is even, $D_i=\lfloor np\rfloor$ for all $i\in [n]$; 
    \item[(iii)] When $np$ is not an integer and $n\lfloor np\rfloor$ is odd, 
                 then $D_i = \lfloor np\rfloor$ for all $i\in [n]$ except one vertex $j$, whose degree is
\begin{equation*}
    D_j=
\begin{cases}
    \lfloor np\rfloor+1,    & \text{if } \{np\}>\frac12; \\
    \lfloor np\rfloor + 1 \text{ or } \lfloor np\rfloor - 1, & \text{if } \{np\}=\frac12;\\
    \lfloor np\rfloor-1,    & \text{if } \{np\}<\frac12.
    \end{cases}
\end{equation*}
\end{itemize}
\end{theorem}

\begin{remark} {\rm
The proof follows directly from the discussion above: 
the objective forces each degree to be as close as possible to $np-\tfrac12$, 
and the graphicality constraints reduce to the parity conditions encoded in the three cases.  
The verification that these degree sequences are realizable follows directly from the Erd\H{o}s-Gallai 
theorem (see, e.g., \cite{Tripathi2003}), 
which gives a necessary and sufficient condition for a non-negative integer sequence to be graphical; 
the details are straightforward and thus omitted for brevity.}   
\end{remark}

Theorem~\ref{theo:frechet} reveals that every Fr\'echet mean of $G_{n,p}$ is a \emph{quasi‑regular} graph
(i.e., all vertex degrees in it differ by at most one).  
This reflects the well‑known homogeneity of the Erd\H{o}s-R\'enyi random graph.  
In particular, case~(ii) asserts that whenever $n\lfloor np\rfloor$ is even, every $\lfloor np\rfloor$‑regular graph (see \cite{bose1963strongly,Bollobas2001}) is a Fr\'echet mean of $G_{n,p}$.
For enumeration results on regular graphs, we refer to 
\cite{Bollobas1980,hofstad2016random}.
Two extreme regimes are worth highlighting:
\begin{itemize}
\item If $np<1$, then $\lfloor np\rfloor=0$ and $n\lfloor np\rfloor=0$ is even, 
so case~(ii) applies and the only possible degree sequence is all zeros.  
Hence the Fr\'echet mean set consists solely of the empty graph.

\item If $n(1-p)<1$ (i.e., $p>1-1/n$), then $np>n-1$.  
Consequently $\lfloor np\rfloor=n-1$ and $n\lfloor np\rfloor=n(n-1)$ is always even, 
so again case~(ii) applies and every vertex must have degree $n-1$.  Thus the only Fr\'echet mean is the complete graph $K_n$.
\end{itemize}
Apart from these two boundary situations, 
the Fr\'echet mean is never unique provided $n$ is sufficiently large (typically $n\ge 4$).  
This non-uniqueness follows from the fact that the degree sequences considered 
in the theorem can correspond to multiple non-isomorphic graphs.

\section{First two moments of the squared Frobenius distance}

The purpose of this paper is to investigate the asymptotic behavior of the Frobenius distance between an
Erd\H{o}s-R\'enyi random graph $G_{n,p}$ and an arbitrary element of its Fr\'echet mean set as $n\to\infty$.
In what follows, we may only focus on the case where $np$ is not an integer and $n\lfloor np\rfloor$ is even,
which satisfies Theorem \ref{theo:frechet}(ii).
Typically, we let $p=p(n)\in (0,1)$ depend on $n$.

For ease of notation, we set $m=\lfloor np\rfloor$ hereafter.
Let $R\in \mathcal{G}_n$ be an arbitrary but fixed $m$-regular graph 
with adjacency matrix $M=(m_{ij})\in \mathcal{A}_n$.
From Theorem \ref{theo:frechet}, the graph $R$ is a Fr\'echet mean of a random graph $G_{n,p}$.
Define the Frobenius distance to this mean as
\begin{equation*}
   F_n=d_{\text{F}}(G_{n,p}, R). 
\end{equation*}
Although the distribution of $F_n$ may depend on $R$, 
we will show that the asymptotic distribution of $F_n^2$, 
as well as its mean and variance, does not depend on the particular choice of the $m$-regular graph $R$.

Recall that $A_{n,p} = (I_{ij})$ and $D_{n,p} = \operatorname{diag}(X_1, X_2,\dots, X_n)$
denote the adjacency and degree matrices of $G_{n,p}$, respectively,
where $I_{ij}$ is the edge indicator and $X_i = \sum_{j \neq i} I_{ij}$.
Since each vertex degree $D_i$ in $R$ is equal to $m$, 
it follows by \eqref{dF2GGnp2} that the squared distance
\begin{align}\label{sqFn}
F_n^2 &= \sum_{i=1}^n \bigg( m^2 + m - 2(m - 1)X_i - 2 \sum_{j \neq i} m_{ij} I_{ij} 
         + \sum_{j\neq i}\sum_{k\neq i,j}I_{ij} I_{ik} \bigg) \notag \\
      &= m(m + 1)n - 2U_n +W_n,
\end{align}
where 
\[
U_n=\sum_{i=1}^n\sum_{j\neq i} (m-1+m_{ij}) I_{ij} \quad \mbox{and} \quad  
W_n=\sum_{i=1}^n\sum_{j\neq i}\sum_{k\neq i,j} I_{ij} I_{ik}.
\]
Note that the sum in $U_n$ counts each undirected edge twice, 
so $U_n/2$ is the total weight of edges in $G_{n,p}$ with weights $m-1+m_{ij}$.
Similarly, $W_n$ counts ordered triples $(i,j,k)$ with $j\ne k$, 
which corresponds to twice the number of wedges
(a wedge is an unordered triple $(i,j,k)$ such that both $ij$ and $ik$ are edges).

From \eqref{sqFn}, one can see that it is more convenient to study the squared distance $F_n^2$ but not $F_n$ directly.
For the first two moments of $F_n^2$, we have the following.

\begin{proposition}\label{meanvarFn2}
For the squared Frobenius distance $F_n^2$,
\begin{align}
 \E[F_n^2] & =(np-m)(np-m-1)n+(3n-2)np(1-p),   \label{EFn2}\\  
 \mathrm{Var}(F_n^2) &=2np(1-p)\big[4(np-m+1-3p)^2n+(n^2+n+18)p(1-p)+(2np-2m-1)^2-5\big]. \label{VarFn2}
\end{align} 
\end{proposition}

\begin{proof}
We derive the mean and variance of $F_n^2$ via the first two moments of random vector $(U_n,W_n)$.

Since $\sum_{j\neq i}m_{ij}=m$ for each $i\in[n]$, 
\begin{align}\label{EUn}
  \E[U_n] &= \sum_{i=1}^n\sum_{j\neq i} (m-1+m_{ij}) p\notag\\
          &= \sum_{i=1}^n[(n-1)(m-1)p+mp]\notag\\
          &=(mn-n+1)np.
\end{align}
Noting that there are $3\binom{n}{3}$ possible wedges in $G_{n,p}$ and each of them occurs 
with probability $p^2$, we have
\begin{align}\label{EWn}
\E[W_n]=2 \cdot 3\binom{n}{3}p^2=n(n-1)(n-2)p^2.
\end{align}
Taking expectations on both sides of \eqref{sqFn} gives
\begin{align*} 
\E[F_n^2]= m(m+1)n - 2\E[U_n] + \E[W_n].
\end{align*}
By \eqref{EUn} and \eqref{EWn}, after straightforward calculations we have
\begin{align*}
 \E[F_n^2] & = m(m+1)n - 2(mn - n + 1)np + n(n-1)(n-2)p^2\\
                   & = \big[(np-m)^2 -(np-m)+(3n-2)p(1-p)\big]n,
\end{align*}
which proves \eqref{EFn2}.

Since each edge indicator $I_{ij}$ appears twice in $U_n$, for the variance of $U_n$ we have
\begin{align}\label{VarUn}
  \mathrm{Var}(U_n) &= 2\sum_{i=1}^n\sum_{j\neq i} (m-1+m_{ij})^2 p(1-p)\notag\\
          &=  2p(1-p)\sum_{i=1}^n\sum_{j\neq i}[(m-1)^2+(2m-1)m_{ij}] \qquad (\text{by }  m_{ij}^2=m_{ij})\notag\\
          &=2np(1-p)[(m-1)^2(n-1)+(2m-1)m]\notag\\
          &=2np(1-p)[n(m-1)^2+m^2+m-1].
\end{align}

For $\mathrm{Var}(W_n)$, note that $W_n/2$ is the number of wedges,
and any two wedge indicators are independent if they do not share common edges.
With a basic combinatorial argument, one can see that in a complete graph on $n$ vertices,
there are $3\binom{n}{3}(2n-5)$ pairs of wedges sharing exactly one edge.  
By symmetry and independence, we thus have
\begin{align*}
\mathrm{Var}\Big(\frac{W_n}2\Big) 
&= 3\binom{n}{3}\cdot\mathrm{Var}(I_{12}I_{13})+2\cdot 3\binom{n}{3}(2n-5) \cdot\mathrm{Cov}(I_{12}I_{13},I_{12}I_{14})\notag\\
&=\frac12n(n-1)(n-2)p^2(1-p^2)+n(n-1)(n-2)(2n-5)p^3(1-p),
\end{align*}
which implies that
\begin{align}\label{VarWn}
 \mathrm{Var}(W_n)=2n(n-1)(n-2)p^2(1-p)[1+(4n-9)p].
\end{align}

Similarly to \eqref{VarWn}, for the covariance we have
\begin{align}\label{CovUnVn}
\mathrm{Cov}(U_n,W_n)
&=2\sum_{i=1}^n\sum_{j\neq i}\sum_{k\neq i,j}\big[\mathrm{Cov}\big((m-1+m_{ij})I_{ij},I_{ij}I_{ik}\big)
   +\mathrm{Cov}\big((m-1+m_{ik})I_{ik},I_{ij}I_{ik}\big)\big]\notag\\
&=4(m-1)\cdot 6\binom{n}{3}\mathrm{Cov}(I_{12},I_{12}I_{13})
       +2\sum_{i=1}^n\sum_{j\neq i}\sum_{k\neq i,j}(m_{ij}+m_{ik})\mathrm{Cov}(I_{12},I_{12}I_{13})\notag\\
&=4(m-1)n(n-1)(n-2)p^2(1-p)+2p^2(1-p)\sum_{i=1}^n\sum_{j\neq i}\sum_{k\neq i,j}(m_{ij}+m_{ik})\notag\\
&=4(m-1)n(n-1)(n-2)p^2(1-p)+4p^2(1-p)mn(n-2)\notag\\
&=4n(n-2)p^2(1-p)(mn-n+1),
\end{align}
where in the ‌penultimate equality we used the following identity: for each $i\in [n]$,
\[
\sum_{j\neq i}\sum_{k\neq i,j}(m_{ij}+m_{ik})=2m(n-2),
\]
which, in fact, is due to that $\sum_{j\neq i}m_{ij}=m$, and 
each term $m_{ij}~ (j\ne i)$ appears $2(n-2)$ times in the double sum.

Also by \eqref{sqFn}, we have 
\begin{align*}
\mathrm{Var}(F_n^2)=4\mathrm{Var}(U_n)+\mathrm{Var}(W_n)-4\mathrm{Cov}(U_n,W_n).
\end{align*}
Substituting \eqref{VarUn}--\eqref{CovUnVn} into this yields
\begin{align*}
 \mathrm{Var}(F_n^2) &=2np(1-p)\big[4(m-1)^2n+4(m^2+m-1)-(8mn-9n+9)(n-2)p\notag\\
                    &\quad +(n-1)(n-2)(4n-9)p^2\big], 
\end{align*}
which, together with some basic algebra, implies \eqref{VarFn2}.
\end{proof}

The first two moments of $F_n^2$ given in \eqref{EFn2} and \eqref{VarFn2} 
lead to the following weak convergence results.

\begin{proposition}\label{Prop:weak}
Let $F_n$ be the Frobenius distance between $G_{n,p}$ and an arbitrary Fr\'echet mean graph.  
\begin{itemize}
\item[(i)] If $n^2p(1-p)\rightarrow 0$, then $F_n\xrightarrow{P}0$.
\item[(ii)] If $n^2p(1-p)\rightarrow \lambda$ for some constant $\lambda>0$, then
\[F_n\xrightarrow{D}2\sqrt{\mathrm{Poi}\Big(\frac{\lambda}2\Big)},\]
where $\mathrm{Poi}(\frac {\lambda}2)$ denotes a Poisson random variable with mean $\frac {\lambda}2$. 
\item[(iii)] If $n^2p(1-p)\rightarrow \infty$, then
\begin{equation}\label{limEFn1}
    \frac{F_n}{\sqrt{\E[F_n^2]}} \xrightarrow{P} 1.
\end{equation}
In particular, if we further have that $np(1-p)$ converges to a non-negative limit or diverges to infinity, 
\begin{equation}\label{limEFn2par}
    \frac{F_n}{\sqrt{n^2p(1-p)}} \xrightarrow{P} \sqrt{a},
\end{equation}
where the constant
\begin{equation}\label{Def:a}
a=\begin{cases}
    2, & \text{if } np(1-p)\to 0;\\
    2+\dfrac{\big(c-\lfloor c\rfloor\big)^2+\lfloor c\rfloor}{c}, &\text{if }  np(1-p)\to c \text{ for some constant } c>0;\\
    3, & \text{if } np(1-p)\to\infty.
\end{cases}
\end{equation}
\end{itemize}
\end{proposition}

\begin{proof}
For any Fr\'echet mean $R$ of $G_{n,p}$, it follows by \eqref{basicpdF} that
\begin{equation*}
F_n=d_{\text{F}}(\overline{G}_{n,p},\overline{R}),
\end{equation*}
where $\overline{R}$ is the complement graph of $R$.
This implies that $\overline{R}$ is also a Fr\'echet mean of $\overline{G}_{n,p}$.
Consequently, without loss of generality we can always assume $0<p\le \frac12$.
Especially, we may let $p\to 0$ in (i) and (ii), 
since if $n^2p(1-p)\to c$ for some constant $c\ge0$, the edge probability $p$ must tend to 0 or 1.
It should also be noted that, under the condition $n^2p\to c<\infty$, 
the Fr\'echet mean $R$ is unique to being empty for sufficiently large $n$
(see the comments below Theorem \ref{theo:frechet}).

When $n^2p\to 0$, we have $mn\to 0$ and $np\to 0$.
It thus follows by \eqref{EFn2} that $\E[F_n^2]$ tends to 0. 
Therefore, the squared distance $F_n^2$ converges in probability to 0, which proves (i).

Consider (ii). Since $R$ is empty for sufficiently large $n$ (i.e., $m=m_{ij}=0$), 
we can simplify \eqref{sqFn} to
\[
F_n^2=4E_n+W_n,
\]
where $E_n=\sum_{1\le i<j\le n}I_{ij}$ denotes the number of edges in $G_{n,p}$. 
It follows by \eqref{EWn} that $\E[W_n]$ tends to 0, 
which implies that the count $W_n$ converges in probability to 0. 
From the known result in \cite[Theorem 3.19]{janson2000}, 
$E_n$ converges in distribution to $\mathrm{Poi}(\lambda/2)$. Thus,
(ii) follows by Slutsky's theorem and the continuous mapping theorem (see, e.g., \cite{Durrett2010}).

To prove \eqref{limEFn1}, it suffices to establish the convergence
\begin{align}\label{weakconv}
\frac{F_n^2}{\E[F_n^2]}\xrightarrow{P}1,
\end{align}
under the condition $n^2p(1-p)\rightarrow \infty$, also by virtue of the continuous mapping theorem.  
To this end, we rely on Proposition \ref{meanvarFn2}
and analyze the orders of magnitude of $\E[F_n^2]$ and $\mathrm{Var}(F_n^2)$, respectively.

From \eqref{EFn2}, the order of the mean can be directly characterized as 
\begin{equation}\label{orderEFn2}
\E[F_{n}^2]=\Theta\big(n^2p(1-p)\big),
\end{equation}
as $n^2p(1-p)\rightarrow \infty$.
In particular, if we further have that $np(1-p)$ converges to a non-negative limit $c\ge0$ or diverges to infinity,  
straightforward calculations gives
\begin{equation}\label{EFnasya}
\frac{\E[F_{n}^2]}{n^2p(1-p)}\to a,
\end{equation}
where the constant $a\in [2,3]$ is defined in \eqref{Def:a}. 

Next, we analyze the variance using \eqref{VarFn2}.
Through basic algebraic simplification and order analysis, we have that as $n^2p(1-p)\rightarrow \infty$,
\begin{align}\label{VarFn22}
 \mathrm{Var}(F_n^2) \sim 2np(1-p)\big[4(np-m+1-3p)^2n+n^2p(1-p)].
\end{align}
By the definition of the integer $m$, we have that the term $np-m+1-3p\in (-2,2)$, 
which bounds the coefficient $4(np-m+1-3p)^2$ by a positive absolute constant.
Under the given condition $n^2p(1-p)\rightarrow \infty$,  
by \eqref{orderEFn2}  there exists an absolute constant $C>0$  such that
\begin{align*}
\frac{\mathrm{Var}(F_n^2)}{(\E[F_n^2])^2} &\le \frac{Cnp(1-p)[n+n^2p(1-p)]}{[n^2p(1-p)]^2}
=\frac{C}{n^2p(1-p)}+\frac{C}{n}\to 0,
\end{align*}
which proves \eqref{weakconv} by Chebyshev's inequality.

Furthermore, the convergence result \eqref{limEFn2par} follows immediately from \eqref{limEFn1} and Slutsky's Theorem.
This completes the proof of Proposition \ref{Prop:weak}.
\end{proof}

\begin{remark} {\rm
From \eqref{VarFn22}, one can conduct a refined order-of-magnitude analysis of the variance $\mathrm{Var}(F_n^2)$
under the condition that $np(1-p) \to c\ge0$
 (for a constant $c$) or $np(1-p)\rightarrow\infty$. 
More precisely, the asymptotic equivalence of the variance is given by
\begin{align}\label{VarFn2asy}
\mathrm{Var}(F_n^2)\sim \begin{cases}
 8n^2p(1-p), &\text{if } np(1-p)\rightarrow 0;\\
 2c\big[4(c-\lfloor c\rfloor+1)^2+c\big]n, &\text{if } np(1-p)\rightarrow c 
    \text{ for some non-integer constant } c>0;\\
 2n^3p^2(1-p)^2, & \text{if }  np(1-p)\rightarrow\infty.
\end{cases}
\end{align}
It is critical to note that when $c$ is a positive integer, 
the ratio $\mathrm{Var}(F_n^2)/n$ does not possess a unique limit, 
which we rigorously verify via the following asymptotic argument.
First, suppose that $np$ is monotonically decreasing and converges to a positive integer $c$. 
Since $np$ is non-integer for all $n\ge1$ by assumption, we have $np>c$ for every $n\ge1$,
and thus $m=c$ for all sufficiently large $n$. 
Substituting $m=c$ into \eqref{VarFn22} and taking the asymptotic limit, we obtain
\[
\frac{\mathrm{Var}(F_n^2)}{n}\to 2c(4+c),
\]
as $n\to\infty$. 
Conversely, if $np$ is monotonically increasing and converges to the positive integer $c$, 
we have $m = \lfloor np \rfloor = c-1$ for all sufficiently large $n$, which yields a distinct limit
 \[\frac{\mathrm{Var}(F_n^2)}{n}\to 2c(16+c).\]
Accordingly, we conclude that $\mathrm{Var}(F_n^2)=\Theta(n)$ when \(c\) is a positive integer, 
but there exists no universal constant $a$ such that \(\mathrm{Var}(F_n^2)\sim an\) as $n\to\infty$.
}    
\end{remark}

\section{Asymptotic normality}
In this section, we establish the asymptotic normality for $F_n^2$ under the condition $n^2p(1-p)\to \infty$. 
Proposition \ref{Prop:weak} implies that when $n^2p(1-p)$ fails to diverge to infinity,
$F_n$ cannot be asymptotically normal under any normalization. 
We first state the main results of this paper in the following.

\begin{theorem}\label{Thm:Fn2}
For the distance $F_n = d_{\mathrm{F}}(G_{n,p}, R)$ with an arbitrary Fr\'echet mean graph $R$, we have that as $n^2p(1-p) \to \infty$,
\begin{equation}\label{Fn2AN1}
 \frac{F_n^2-\E[F_n^2]}{\sqrt{\mathrm{Var}(F_n^2)}} \xrightarrow{D} N(0,1).    
\end{equation}
In particular, the following assertions hold.
\begin{itemize}
     \item[(i)] If $np(1-p)\to 0$ and $n^2p(1-p)\to \infty$, then
    \[
    \frac{F_n^2-\E[F_n^2]}{\sqrt{n^2p(1-p)}} \xrightarrow{D} N(0,8). 
    \]
    \item[(ii)] If $np(1-p)\to c$ for some non-integer constant $c>0$, then
    \[
    \frac{F_n^2-\E[F_n^2]}{\sqrt{n}} \xrightarrow{D} 
    N\big(0, 2c\big[4(c-\lfloor c\rfloor+1)^2+c\big]\big).
    \]
    \item[(iii)] If $np(1-p)\to \infty$, then
    \[
    \frac{F_n^2-\E[F_n^2]}{\sqrt{n^3p^2(1-p)^2}} \xrightarrow{D} N(0,2). 
    \]
\end{itemize}
\end{theorem}

Theorem \ref{Thm:Fn2} reveals a  phase transition in the asymptotic behavior of $F_n^2$,
characterized by the scaling of $np(1-p)$, 
which determines the appropriate normalization and limiting variance.
By the symmetry \eqref{basicpdF}, 
without loss of generality we can assume $p\le \frac12$ in the proof of the theorem. 
Then, these three special regimes correspond, respectively, 
to very sparse regime ($np\to0$ but $n^2p\to\infty$), 
sparse  regime ($np\to c>0$), 
and dense regime ($np\to\infty$) for Erd\H{o}s-R\'enyi random graphs.
When $c$ is a positive integer,
the asymptotic normality of the above form requires additional constraints, 
as indicated in Remark 2; for instance, 
$np$ should converge to $c$ monotonically from above or below. 
Following the existing results, 
the precise formulation of the corresponding conclusions can be readily derived, 
and thus we omit the details here.

Applying the delta method (see, e.g., \cite[Chapter 3]{vaart1998}) to the results of Theorem \ref{Thm:Fn2}, 
we obtain the following asymptotic distributions for the Frobenius distance itself.

\begin{theorem}\label{Thm:Frobenius}
Under the same conditions of Theorem \ref{Thm:Fn2}, we have the following.
\begin{itemize}
     \item[(i)] If $np(1-p)\to 0$ and $n^2p(1-p)\to \infty$, then
    \[
    F_n-\sqrt{\E[F_n^2]} \xrightarrow{D} N(0,1). 
    \]
    \item[(ii)] If $np(1-p)\to c$ for some non-integer constant $c>0$, then
    \[
    F_n-\sqrt{\E[F_n^2]} \xrightarrow{D} 
    N\left(0, \frac{c\big[4(c-\lfloor c\rfloor+1)^2+c\big]}
    {2\big[(c-\lfloor c\rfloor)^2+\lfloor c\rfloor+2c\big]}\right).
    \]
    \item[(iii)] If $np(1-p)\to \infty$, then
    \[
    \frac{F_n-\sqrt{\E[F_n^2]}}{\sqrt{np(1-p)}} \xrightarrow{D} 
    N\Big(0,\frac16\Big). 
    \]
\end{itemize}
\end{theorem}

\subsection{Proof of Theorem \ref{Thm:Fn2}(i) and (ii)}
\label{poofthm2}

This subsection and the one that follows are devoted to the proof of Theorem \ref{Thm:Fn2}. 
Throughout the entire proof, we assume without loss of generality that \(0<p\le \frac{1}{2}\). 
By a standard subsequence argument, 
it suffices to verify that the convergence \eqref{Fn2AN1} holds under each of the conditions (i), (ii), 
and (iii) to complete the proof of Theorem \ref{Thm:Fn2}.

We adopt different techniques to handle the three regimes in Theorem~\ref{Thm:Fn2}:
\begin{itemize}
    \item When $np(1-p)\to c$ for some constant $c\ge0$, 
    the random variables contributing to $F_n^2$ exhibit a sparse dependency structure. 
    In these cases, we construct a suitable dependency graph \cite{Baldi1989a,Baldi1989b} for the summands and 
    verify the conditions of Lemma~\ref{Lem:dependency} to establish asymptotic normality.
    \item When $np(1-p)\to\infty$, 
    the dependencies become too intricate for a simple dependency graph approach. 
    Instead, we employ recent Berry–Esseen bounds for functionals of independent random variables developed 
    by Shao and Zhang \cite{shao2025104574}, which provide sharper control on the convergence rate.
\end{itemize}
The proof for (i) and (ii) using dependency graphs is presented first; 
the proof for (iii) using Berry–Esseen bounds is given in the next subsection.

Let us first recall some fundamental concepts of the dependency graph for random variables. 
Let $\{X_i\}_{i\in \mathcal{I}}$ be a family of random variables defined on a common probability space.
A {\em dependency graph} for these random variables is any graph $L$ with vertex set $\mathcal{I}$ such that 
if $A$ and $B$ are two disjoint subsets of $\mathcal{I}$ with no edges between $A$ and $B$, 
then the families $\{X_i\}_{i\in A}$ and $\{X_j\}_{j\in B}$ are mutually independent.
For any positive integer $r$ and $i_1,i_2,\dots,i_r\in \mathcal{I}$, we also denote by
\[
\overline{N}_L(i_1,i_2,\dots,i_r)=\bigcup_{k=1}^r\{j\in\mathcal{I}:j=i_k \text{ or } 
j \text{ is adjacent to } i_k \text{ in } L\}
\]
the closed neighborhood of $\{i_1,i_2,\dots,i_r\}$ in $L$.

The following auxiliary lemma, stated as Theorem 6.33 in \cite{janson2000}, 
plays an important role in our proof of the asymptotic normality of $F_n^2$.

\begin{lemma}\label{Lem:dependency}
Suppose that $\{S_n\}_{n=1}^{\infty}$ is a sequence of random variables such that $S_n=\sum_{\alpha\in {\cal A}_n}X_{n\alpha}$, 
where for each $n$, $\{X_{n\alpha},\alpha\in {\cal A}_n\}$ is a family of random variables with dependency graph $L_n$.
Suppose further that there exist numbers $Q_{1n}$ and $Q_{2n}$ such that 
$\sum_{\alpha\in {\cal A}_n}\E[|X_{n\alpha}|]\leq Q_{1n}$ and, 
for every $\alpha_1,\alpha_2\in {\cal A}_n$,
\[
\sum_{\alpha\in\overline{N}_{L_n}(\alpha_1,\alpha_2)}\E[|X_{n\alpha}|\big|X_{n\alpha_1},X_{n\alpha_2}]\leq Q_{2n}.
\]
As $n\rightarrow \infty$, if $Q_{1n}Q_{2n}^2/(\mathrm{Var}(S_n))^{3/2}\rightarrow 0$, then
\[
\frac{S_n-\E[S_n]}{\sqrt{\mathrm{Var}(S_n)}} \xrightarrow{D} N(0,1).
\]
\end{lemma}

\begin{proof}[Proof of (i) and (ii) in Theorem \ref{Thm:Fn2}]
Recall from \eqref{sqFn} that $F_n^2$ can be expressed as a linear combination of edge and wedge indicators.
Let us denote by $\{T_{\alpha}\}_{\alpha\in {\cal B}_n}$ the set of all possible edges and wedges
in $G_{n,p}$, where ${\cal B}_n$ stands for an index set with the cardinality 
$|{\cal B}_n|=\binom{n}{2}+3\binom{n}{3}=\frac12n(n-1)^2$.
For each $\alpha\in {\cal B}_n$, define
\[
Y_{n\alpha}= \begin{cases}
    -4(m-1+m_{ij})\,I_{ij}, & \text{if $T_{\alpha}$ is the edge between vertices $i$ and $j$};\\[2mm]
    2\,I_{ij}I_{ik}, & \text{if $T_{\alpha}$ is the wedge centered at $i$ with neighbors $j\neq k$}.
\end{cases}
\]
Then, by \eqref{sqFn},
\begin{equation}\label{def:wtFn2}
  \widetilde{F}_n^2:= F_n^2-m(m+1)n=\sum_{\alpha\in {\cal B}_n}Y_{n\alpha}.
\end{equation}
Note that $\mathrm{Var}(\widetilde{F}_n^2)=\mathrm{Var}(F_n^2)$.
To prove the desired result \eqref{Fn2AN1} in Theorem \ref{Thm:Fn2}, it is now sufficient to show that
\begin{equation}\label{tildeFn2AN}
    \frac{\widetilde{F}_n^2-\E[\widetilde{F}_n^2]}{\sqrt{\mathrm{Var}(\widetilde{F}_n^2)}} \xrightarrow{D} N(0,1).
\end{equation}

We will apply Lemma \ref{Lem:dependency} to the collection $\{Y_{n\alpha}\}_{\alpha\in{\cal B}_n}$. 
First, construct a dependency graph $L_n$ with vertex set ${\cal B}_n$ by connecting two vertices $\alpha,\beta$ with an edge 
if and only if the corresponding subgraphs $T_{\alpha}$ and $T_{\beta}$ share at least one edge. 
(If they share no edge, the corresponding random variables are independent because they involve disjoint sets of edge indicators.)

Now we verify the conditions of Lemma \ref{Lem:dependency}. 
Without loss of generality, as in the proof of Proposition \ref{Prop:weak}, 
we may assume $np\to c$ for some constant $c\ge0$, 
so that $p\to0$ and $m=\lfloor np\rfloor$ satisfies $|m-1+m_{ij}|\le c+1$ for large $n$.

Since $|Y_{n\alpha}|\le 4(c+1)$ for an edge and $|Y_{n\alpha}|\le 2$ for a wedge,
\[
\sum_{\alpha\in{\cal B}_n}\E[|Y_{n\alpha}|] \le 4(c+1)\binom{n}{2}p + 2\cdot 3\binom{n}{3}p^2 \le 3(c+1)n^2p := Q_{1n},
\]
which is of order $n^2p$.

Fix any two vertices $\alpha_1,\alpha_2\in{\cal B}_n$ and 
let $\overline{N}_{L_n}(\alpha_1,\alpha_2)$ be their closed neighborhood in $L_n$. 
Note that the union $T_{\alpha_1}\cup T_{\alpha_2}$ contains at most six vertices. 
It follows that  the number of edges in this union is at most $\binom{6}{2}$,
and that the number of edges wedges that intersect this union is 
at most $3\binom{6}{3}+2\binom{6}{2}(n-6)$.   
These simple facts implies
\[
\sum_{\alpha\in \overline{N}_{L_n}(\alpha_1,\alpha_2)}\E[|Y_{n\alpha}|\given Y_{n\alpha_1},Y_{n\alpha_2}] 
\le \binom{6}{2}+3\binom{6}{3}+2\binom{6}{2}(n-6)p.
\]
Thus, we may take $Q_{2n}=O(1)$, since $np$ is bounded in Cases (i) and (ii). 

By \eqref{VarFn2asy} and \eqref{def:wtFn2}, we have $\mathrm{Var}(\widetilde{F}_n^2)=\mathrm{Var}(F_n^2)$ of order $n^2p$ 
in Case (i) and of order $n$ in Case (ii). Consequently,
\[
\frac{Q_{1n}Q_{2n}^2}{(\mathrm{Var}(\widetilde{F}_n^2))^{3/2}} = 
\begin{cases}
O\!\left(\dfrac{n^2p}{(n^2p)^{3/2}}\right)=O\!\left(\dfrac1{\sqrt{n^2p}}\right)\to0, & \text{in Case (i)},\\[4mm]
O\!\left(\dfrac{n^2p}{n^{3/2}}\right)=O\!\left(\dfrac1{\sqrt{n}}\right)\to0, & \text{in Case (ii)}.
\end{cases}
\]
Thus the condition of Lemma \ref{Lem:dependency} is satisfied, and \eqref{tildeFn2AN} holds for $np\to c$
for some constant $c\ge0$. 

If $np$ tends to 0 or a non-integer constant $c>0$,  
then (i) and (ii) of Theorem \ref{Thm:Fn2} follows by \eqref{Fn2AN1} and Slutsky's theorem.
\end{proof}

We next consider (iii), and begin with some basic analysis. 
Note that we now have $0<p\le \frac12$ and $np\to\infty$, which is the defining regime of (iii).

Recall that $E_n$ denotes the number of edges in $G_{n,p}$, and $W_n$ is twice the number of wedges.
By \eqref{sqFn} and \eqref{def:wtFn2}, we can rewrite $\widetilde{F}_n^2$ as 
\begin{equation}\label{wtFn22}
\widetilde{F}_n^2=-2\widetilde{U}_n-4(n-2)pE_n+W_n
\end{equation}
where 
\[
\widetilde{U}_n=\sum_{i=1}^n\sum_{j\neq i}(m-np-1+m_{ij}+2p)I_{ij}.
\]
Here we take $4(n-2)p$ as the coefficient of $E_n$ instead of $4np$,
since this choice will make the subsequent computations more concise.
Note that these coefficients satisfy 
\[|m-np-1+m_{ij}+2p|\le 2,\]
since $0< np-m<1, m_{ij}\in\{0,1\}$, and $0<p\le \frac12$. 

Now $\mathrm{Var}(\widetilde{F}_n^2)=\mathrm{Var}(F_n^2)$ by \eqref{def:wtFn2}, and from \eqref{VarFn2asy} we have, as $np\to\infty$,
\begin{equation}\label{VarHn}
\mathrm{Var}(\widetilde{F}_n^2) \sim 2n^3p^2(1-p)^2.    
\end{equation}
A direct computation gives
\begin{equation}\label{VartUn}
\mathrm{Var}(\widetilde{U}_n) \le  16\,\mathrm{Var}(E_n)=16\,\binom{n}{2}p(1-p)= O(n^2p)
=o\big(\mathrm{Var}(\widetilde{F}_n^2)\big),
\end{equation}
since $np\to\infty$. Consequently,
\[
\frac{2(\widetilde{U}_n-\E[\widetilde{U}_n])}{\sqrt{\mathrm{Var}(\widetilde{F}_n^2)}} \xrightarrow{P} 0
\]
by Chebyshev's inequality. 
Moreover, from \eqref{wtFn22} and Slutsky's theorem, to prove \eqref{tildeFn2AN} it suffices to show that
\begin{equation}\label{Vn4npEn}
\frac{(W_n-4(n-2)pE_n)-\E[W_n-4(n-2)pE_n]}{\sqrt{2n^3p^2(1-p)^2}} \xrightarrow{D} N(0,1).
\end{equation}

We now examine why the dependency graph method used for (i) and (ii) cannot be directly applied to establish \eqref{Vn4npEn}. 
Under the assumptions $p\le 1/2$ and $np\to\infty$, it is known that $E_n$ and $W_n$ satisfy a joint central limit theorem:
\[
\sqrt{\frac{2p}{1-p}}\begin{pmatrix}
\displaystyle\frac{E_n-\E[E_n]}{np} \\[6pt]
\displaystyle\frac{W_n-\E[W_n]}{4n^2p^2}
\end{pmatrix}
\xrightarrow{D} \begin{pmatrix} \mathcal{N} \\ \mathcal{N} \end{pmatrix},
\]
where $\mathcal{N}$ denotes a standard normal random variable 
(see \cite[Theorem 3(iii)]{feng2013} and \cite[Theorem 1(ii)]{feng2025limit}; 
note that both results are obtained via dependency graph arguments).
Since the limit in distribution is a degenerate normal, 
after straightforward calculations it follows that the linear combination $W_n-4(n-2)pE_n$ satisfies
\[
\frac{W_n-4(n-2)pE_n - \E[W_n-4(n-2)pE_n]}{\sqrt{n^4p^3}} \xrightarrow{P} 0,
\]
which indicates that the joint convergence is degenerate for this particular linear functional;
the asymptotic normality of $W_n-4(n-2)pE_n$ requires a finer analysis that captures the next‑order fluctuations.
Consequently, the dependency graph approach does not provide sufficient precision to prove \eqref{Vn4npEn},
and a more sophisticated method—such as the Berry–Esseen bounds for functionals of independent variables developed in \cite{shao2025104574}—is needed. This justifies our separate treatment of Theorem \ref{Thm:Fn2}(iii).

\subsection{Proof of Theorem \ref{Thm:Fn2}(iii)}

Given the limitations of the dependency graph approach, 
we now turn to a more refined normal approximation technique. 
Recently, Shao and Zhang \cite{shao2025104574} developed a powerful method for establishing Berry–Esseen bounds 
for functionals of independent random variables. 
Their approach is based on a discrete-difference version of Stein's method, 
inspired by the generalized perturbative framework of Chatterjee \cite{chatterjee2008new}. 
This method yields a bound on the Kolmogorov distance between the distribution of a standardized statistic 
and the standard normal distribution. To prepare for the proof of (iii), 
we first recall the key result from \cite{shao2025104574} that will be instrumental in our analysis.

Let $N\ge 2$ be an integer, $\bm{X}=(X_1,\dots,X_N)$ a vector of independent random variables
taking values in a measurable space $\mathcal{X}$, 
and $\bm{X}'=(X_1',\dots,X_N')$ an independent copy. 
For any subset $A \subset [N]$, 
define $\bm{X}^A$ by replacing $X_s$ with $X_s'$ for $s\in A$ and leaving others unchanged. For a measurable function $h:\mathcal{X}^N\to\mathbb{R}$, set
\begin{equation}\label{defdiffer}
\Delta_s h(\bm{X}^A)=h(\bm{X}^A)-h(\bm{X}^{A\cup\{s\}}),\quad s\in[N].
\end{equation}
In particular, $\Delta_s h(\bm{X})=\Delta_s h(\bm{X}^{\varnothing})$.

Consider the standardized random variable
\[
H = \frac{h(\bm{X}) - \E[h(\bm{X})]}{\sigma},
\]
where $\sigma^2={\rm Var}(h(\bm{X}))$. 
The following lemma, stated as Theorem 2.1 of \cite{shao2025104574}, 
provides a bound on the Kolmogorov distance between $H$ and the standard normal random variable $\mathcal{N}$.

\begin{lemma}\label{Shaotheorem}
Let $A_s = \{1, 2, \dots, s-1\}$ for $s \geq 2$ and $A_1 = \varnothing$. Define random variables
\[
V = \frac{1}{2\sigma^2} \sum_{s=1}^N \Delta_s h(\bm{X}) \Delta_s h(\bm{X}^{A_s}) \quad \text{and} \quad 
V^* = \frac{1}{\sigma^2} \sum_{s=1}^N \Delta_s h(\bm{X}) \big|\Delta_s h(\bm{X}^{A_s})\big|.
\]
Then the Kolmogorov distance between $H$ and $\mathcal{N}$ satisfies
\[
d_{\rm K}(H, \mathcal{N})\leq\E\left|1 -\E[V\given H]\right|+2\E \left| \E[V^* \given H]\right|.
\]
\end{lemma}

To apply Lemma \ref{Shaotheorem}, we encode the random graph model as a random vector of edge indicators. 
Let $N=\binom{n}{2}$ be the total number of possible edges in $G_{n,p}$. 
We fix an ordering of the edges by sorting pairs $(i,j)$ lexicographically: 
first by increasing $i$, then by increasing $j$. 
More precisely, for $s\in[N]$ define
\begin{equation*}
i(s)=\min\left\{1\le t\le n-1:s\le \frac{(2n-t-1)t}{2}\right\},\quad
j(s)=s+i(s)-\frac{[2n-i(s)][i(s)-1]}{2},
\end{equation*}
which gives a bijection between $[N]$ and the set of all possible edges in $G_{n,p}$.
Consequently, the random vector $\bm{Z}=(Z_1,\dots,Z_N)$ with $Z_s=I_{i(s)j(s)}$ represents the edge indicators of $G_{n,p}$.

Both $E_n$ and $W_n$ can be viewed as functions of $\bm{Z}$; 
we will occasionally write $E_n(\bm{Z})$ and $W_n(\bm{Z})$ to emphasize this dependence. 
Define 
\begin{equation}\label{defh}
    h(\bm{Z})=\frac12W_n(\bm{Z})-2(n-2)pE_n(\bm{Z})=\frac12W_n-2(n-2)pE_n,
\end{equation}
i.e., the number of wedges minus $2(n-2)p$ times the number of edges.
Note that $2h(\bm{Z})=W_n-4(n-2)pE_n$, 
so proving asymptotic normality of $h(\bm{Z})$ is equivalent to establishing \eqref{Vn4npEn}.
Let $\sigma_n^2 = \operatorname{Var}[h(\bm{Z})]$. 
From \eqref{wtFn22}-\eqref{VartUn}, we obtain
\begin{equation}\label{sigman2}
\sigma_n^2 = \frac12 n^3 p^2 (1-p)^2 (1+o(1)).
\end{equation}
We then introduce the standardized statistic
\begin{equation*}
H_n = \frac{h(\bm{Z}) - \E[h(\bm{Z})]}{\sigma_n}.
\end{equation*}
Our goal is to prove that under the condition that $p\le \frac12$ and $np\to \infty$,
\begin{equation}\label{HnAN}
H_n \xrightarrow{D} N(0,1).
\end{equation}

To define the discrete difference operator required in Lemma \ref{Shaotheorem}, 
we need an independent copy of $\bm{Z}$. 
Let $\{I_{ij}':1\le i<j\le n\}$ be an independent copy of $\{I_{ij}:1\le i<j\le n\}$, 
and set $\bm{Z}'=(Z_1',\dots,Z_N')$ with $Z_s' = I_{i(s)j(s)}'$. 
Then $\bm{Z}'$ is an independent copy of $\bm{Z}$. 
For any subset $A\subset[N]$, 
denote by $\bm{Z}^A$ the vector obtained from $\bm{Z}$ by replacing $Z_s$ with $Z_s'$ for $s\in A$ 
and leaving the other coordinates unchanged. 
Following \cite{shao2025104574}, we define the discrete difference
\[
\Delta_s h(\bm{Z}^A) = h(\bm{Z}^A) - h(\bm{Z}^{A\cup\{s\}}),\qquad s\in[N].
\]
In particular, $\Delta_s h(\bm{Z}) = h(\bm{Z}) - h(\bm{Z}^{\{s\}})$. 
Now set $A_s = \{1,\dots,s-1\}$ for $s\ge2$ and $A_1=\varnothing$. 
Then, in accordance with Lemma \ref{Shaotheorem}, we define
\begin{equation}\label{expreVn}
V_n = \frac{1}{2\sigma_n^2}\sum_{s=1}^N \Delta_s h(\bm{Z})\,\Delta_s h(\bm{Z}^{A_s}),
\end{equation}
and
\begin{equation}\label{expreVn*}
V_n^* = \frac{1}{\sigma_n^2}\sum_{s=1}^N \Delta_s h(\bm{Z})\,\bigl|\Delta_s h(\bm{Z}^{A_s})\bigr|.
\end{equation}
Lemma \ref{Shaotheorem} then yields the following bound on the Kolmogorov distance between $H_n$ and a standard normal variable $\mathcal{N}$:
\begin{equation}\label{DKVnVnstar}
d_{\rm K}(H_n,\mathcal{N}) \le \E\bigl|1-\E[V_n\given H_n]\bigr| + 2\E\bigl|\E[V_n^*\given H_n]\bigr|.
\end{equation}
In the next step, we will estimate the right‑hand side and show that it tends to zero as $n\to\infty$, thereby establishing \eqref{HnAN}.

We now do further analysis on $V_n$ and $V_n^*$, especially using their explicit expressions.
For any fixed vertex pair $1\le i<j\le n$, we define a set of edge variables $\{I_{k\ell}^{(i,j)}:1\le k\ne \ell\le n\}$,
where
\begin{equation}\label{Iklij}
I_{k\ell}^{(i,j)}= \begin{cases} I_{k\ell}', & \text{if } \min\{k,\ell\}<i \text{ or } \min\{k,\ell\}=i<\max\{k,\ell\}<j, \\
                 I_{k\ell}, & \text{otherwise}. \end{cases}
\end{equation}
In other words, $I_{k\ell}^{(i,j)}$ equals the independent copy if the edge $(k,\ell)$ 
appears before $(i,j)$ in the ordering, and equals the original indicator otherwise. 
This precisely captures the effect of the replacement set $A_s$.

\begin{lemma}\label{lemma3}
For any fixed $s \in [N]$ with the corresponding vertex pair $(i,j) = (i(s),j(s))$, we have
\begin{equation}\label{deltasz}
  \Delta_s h(\bm{Z}) = \sum_{k \ne i,j} M_{ijk},   
\end{equation}
and for $A_s = \{1,2,\dots,s-1\}$ with $A_1 = \varnothing $,
\begin{equation}\label{deltaszs}
\Delta_s h(\bm{Z}^{A_s}) = \sum_{\ell\ne i,j} \widetilde{M}_{ij\ell}, 
\end{equation}
where
\begin{align}\label{MMbar}
M_{ijk} = \big(I_{ij} - I'_{ij}\big) (I_{ik} +I_{jk}-2p), \quad 
\widetilde{M}_{ij\ell} = \big(I_{ij} - I'_{ij}\big) \big( I_{i\ell}^{(i,j)} +  I_{j\ell}^{(i,j)}-2p \big). 
\end{align}
\end{lemma}
\begin{proof}
For any subset $A\subset [N]$, by \eqref{defdiffer} and \eqref{defh} we have 
\begin{align}\label{deltaih}
\Delta_sh(\bm{Z}^A)
    &=\frac12 W_n(\bm{Z}^A)-2(n-2)pE_n(\bm{Z}^A)-\frac12 W_n(\bm{Z}^{A\cup \{s\}})+2(n-2)pE_n(\bm{Z}^{A\cup \{s\}})\nonumber\\
    &=\frac12\Delta_s W_n(\bm{Z}^A)-2(n-2)p\Delta_s E_n(\bm{Z}^A).
\end{align}

If $A=\varnothing$, it follows by definition that
\begin{align*}
 \Delta_s W_n(\bm{Z}) = 2\sum_{k\neq i,j}(I_{ij}-I_{ij}')(I_{ik}+I_{jk}),   
\end{align*}
and
\[
\Delta_s E_n(\bm{Z}) = I_{ij}-I_{ij}'.
\]
Thus, combining \eqref{deltaih} yields 
\begin{align*}
\Delta_sh(\bm{Z})&=\sum_{k \ne i,j}\big(I_{ij}-I_{ij}'\big)(I_{ik}+I_{jk})-2(n-2)p\big(I_{ij}-I_{ij}'\big)\\
                 &=\sum_{k \ne i,j}\big(I_{ij}-I_{ij}'\big)(I_{ik}+I_{jk}-2p),
\end{align*}
which proves \eqref{deltasz}.

For $A_s=\{1,\dots,s-1\}$, also by definition we have
\begin{align*}
        \Delta_sW_n(\bm{Z}^{A_s})&=2\sum_{\ell \ne i,j}\big(I_{ij}-I_{ij}'\big)(I_{i\ell}^{(i,j)}+I_{j\ell}^{(i,j)}),
\end{align*}
and
\begin{align*}
\Delta_sE_n(\bm{Z}^{A_s})&=I_{ij}-I_{ij}'.
\end{align*}
By \eqref{deltaih}, similarly we have that for any $s\in [N]$,
\begin{align*}
\Delta_sh(\bm{Z}^{A_s})&=\sum_{\ell\ne i,j}\big(I_{ij}-I_{ij}'\big)\big(I_{i\ell}^{(i,j)}+I_{j\ell}^{(i,j)}\big)
         -2(n-2)p\big(I_{ij}-I_{ij}'\big)\\
    &=\sum_{\ell \ne i,j}\big(I_{ij}-I_{ij}'\big)\big(I_{i\ell}^{(i,j)}+I_{j\ell}^{(i,j)}-2p\big),
\end{align*}
and thus \eqref{deltaszs} holds. The proof of Lemma \ref{lemma3} is complete.
\end{proof}

By \eqref{expreVn} and \eqref{expreVn*}, 
Lemma \ref{lemma3} provides the explicit expressions for $V_n$ and $V_n^*$ as follows:
\begin{align}
V_n&=\frac{1}{2\sigma_n^2} \sum_{1\le i<j\le n} \Biggl( \sum_{k\neq i,j} M_{ijk} \Biggr)
\Biggl( \sum_{\ell\neq i,j} \widetilde{M}_{ij\ell} \Biggr),\label{VnEF}\\
V_n^* &= \frac{1}{\sigma_n^2} \sum_{1\le i<j\le n} \sum_{k\neq i,j} M_{ijk}
\Biggl| \sum_{\ell\neq i,j} \widetilde{M}_{ij\ell} \Biggr|.
\label{VnstarEF}
\end{align}

We now discuss the first two moments of $V_n$ and $V_n^*$, and begin by showing
\begin{equation}\label{EVnVnstar}
    \E[V_n]=1,\quad \E[V_n^*]=0.
\end{equation}

Observe that $h(\bm{Z})$ and $h(\bm{Z}')$ are independent and identically distributed, and by telescoping,
\begin{equation}\label{hZhZpdiff}
h(\bm{Z}) - h(\bm{Z}') = h(\bm{Z}^{A_1}) - h(\bm{Z}^{A_{N+1}})=\sum_{s=1}^N \Delta_s h(\bm{Z}^{A_s}),    
\end{equation}
where $A_1=\varnothing$ and $A_{N+1}=[N]$.

Consider the effect of swapping the $s$-th coordinate between $\bm{Z}$ and its copy $\bm{Z}'$. 
This exchange leaves the joint distribution unchanged and transforms 
$\Delta_s h(\bm{Z}^{A_s})$ into $-\Delta_s h(\bm{Z}^{A_s})$, 
while $h(\bm{Z})$ becomes $h(\bm{Z}^{\{s\}})$. 
Hence we have the distributional identity
\begin{equation*}
h(\bm{Z}) \Delta_s h(\bm{Z}^{A_s})\overset{d}{=} -h(\bm{Z}^{\{s\}})\Delta_s h(\bm{Z}^{A_s}),
\end{equation*}
and consequently,
\begin{equation} \label{EhZ}
    \E\big[ h(\bm{Z}) \Delta_s h(\bm{Z}^{A_s})\big]=-\E\big[h(\bm{Z}^{\{s\}})\Delta_s h(\bm{Z}^{A_s})\big].
\end{equation}

Now compute the variance of $h(\bm{Z})$ using the decomposition \eqref{hZhZpdiff}:
\[
\sigma_n^2 = \E\big[h(\bm{Z})(h(\bm{Z})-h(\bm{Z}'))\big] = \sum_{s=1}^N \E\big[h(\bm{Z}) \Delta_s h(\bm{Z}^{A_s})\big].
\]
Applying \eqref{EhZ} and noting that $\Delta_s h(\bm{Z}) = h(\bm{Z}) - h(\bm{Z}^{\{s\}})$, we obtain
\[
2\E\big[h(\bm{Z}) \Delta_s h(\bm{Z}^{A_s})\big] = \E\big[\Delta_s h(\bm{Z}) \Delta_s h(\bm{Z}^{A_s})\big].
\]
Summing over $s$ yields
\[
\sigma_n^2 = \frac12 \sum_{s=1}^N \E\big[\Delta_s h(\bm{Z}) \Delta_s h(\bm{Z}^{A_s})\big].
\]
Taking expectations in the definition \eqref{expreVn} of $V_n$ gives $\E[V_n]=1$.

Next, we show $\E[V_n^*]=0$. 
Recall the explicit expressions \eqref{VnstarEF} and \eqref{MMbar}.
Notice that the factor $(I_{ij}-I_{ij}')$ is common to both $M_{ijk}$ and $\widetilde{M}_{ij\ell}$, 
and it is independent of all other random variables. 
Moreover,
\[
M_{ijk} \bigg|\sum_{\ell\ne i,j} \widetilde{M}_{ij\ell}\bigg| = (I_{ij}-I_{ij}')\big(I_{ik}+I_{jk}-2p\big)\cdot |I_{ij}-I_{ij}'| 
\cdot \bigg|\sum_{\ell\ne i,j} \bigl(I_{i\ell}^{(i,j)}+I_{j\ell}^{(i,j)}-2p\bigr)\bigg|.
\]
Using $(I_{ij}-I_{ij}')|I_{ij}-I_{ij}'| = I_{ij}-I_{ij}'$, this simplifies to
\[
M_{ijk}\bigg|\sum_{\ell\ne i,j} \widetilde{M}_{ij\ell}\bigg| = (I_{ij}-I_{ij}') \big(I_{ik}+I_{jk}-2p\big)
\bigg|\sum_{\ell\ne i,j} \bigl(I_{i\ell}^{(i,j)}+I_{j\ell}^{(i,j)}-2p\bigr)\bigg|.
\]
Since $(I_{ij}-I_{ij}')$ is independent of the remaining factors
and $\E[I_{ij}-I_{ij}'] = 0$, 
taking expectation gives
\[
\E\bigg[ M_{ijk}\bigg|\sum_{\ell\ne i,j}\widetilde{M}_{ij\ell}\bigg|\bigg]=0.
\]
Summing over all $i,j,k$ yields $\E[V_n^*]=0$, completing the proof of \eqref{EVnVnstar}.

For the variances of $V_n$ and $V_n^*$, we have the following result.

\begin{lemma}\label{upperbound}
If $p\le \frac12$ and $np\to\infty$, we have
\[
\max\big\{\mathrm{Var}(V_n), \mathrm{Var}(V^*_n)\big\}=O\Big(\frac1n\Big).
\]
\end{lemma}

\begin{proof}
Let us perform some preliminary computations that will be helpful later. 
For any $1\le i<j\le n$ and $k\ne i,j$, let
\[
J_{ij}=I_{ij} - I'_{ij}, \quad 
J_{ijk}^{(1)}=I_{ik}+I_{jk}-2p, \quad
J_{ijk}^{(2)}=I_{ik}^{(i,j)}+I_{jk}^{(i,j)}-2p.
\]
It is clear that $J_{ij}$ is independent of $J_{ijk}^{(1)}$ and $J_{ijk}^{(2)}$,
\[\E[J_{ij}]=\E\big[J_{ijk}^{(1)}\big]=\E\big[J_{ijk}^{(2)}\big]=0, \quad \E[J_{ij}^2]=2p(1-p),\]
and by \eqref{MMbar},
\[
M_{ijk}=J_{ij}J_{ijk}^{(1)}, \quad   \widetilde{M}_{ijk} =J_{ij}J_{ijk}^{(2)}.
\]
Moreover, for any $k,\ell\notin \{i,j\}$ we have
\begin{equation*}
 \E\big[J_{ijk}^{(1)}J_{ij\ell}^{(2)}\big]
 = {\rm Cov}\big(I_{ik}+I_{jk}, I_{i\ell}^{(i,j)}+I_{j\ell}^{(i,j)}\bigr),  
\end{equation*}
which relies on the relative sizes of the four integers $i,j,k$ and $\ell$. 
By \eqref{Iklij}, if $k\ne\ell$ or $k=\ell<i$, 
the indicators $I_{ik}$ and $I_{i\ell}^{(i,j)}$ are independent, 
as are $I_{jk}$ and $I_{j\ell}^{(i,j)}$.  
Similarly, if $i<k=\ell<j$, it follows that $I_{i\ell}^{(i,j)}=I_{ik}'$  and $I_{j\ell}^{(i,j)}=I_{jk}$;
if $j<k=\ell$, we have $I_{i\ell}^{(i,j)}=I_{ik}$  and $I_{j\ell}^{(i,j)}=I_{jk}$.
Combining these facts, we thus have
\begin{equation}\label{EJ12kl}
 \E\big[J_{ijk}^{(1)}J_{ij\ell}^{(2)}\big]
 = \begin{cases}
 0,       & \text{ if }  k\ne \ell \text{ or } k=\ell<i;\\
 p(1-p),  & \text{ if }  i<k=\ell<j;\\
 2p(1-p),  & \text{ if }  j<k=\ell.  
 \end{cases}
\end{equation}

We first consider the variance of \(V_n\). 
It follows by \eqref{VnEF} that
\begin{align}\label{VarVn}
\mathrm{Var}(V_n) &= \frac{1}{4\sigma_n^4} 
\sum_{i_1<j_1} \sum_{i_2<j_2} 
\mathrm{Cov}\big(\varLambda_{i_1j_1},\varLambda_{i_2j_2}\big), 
\end{align}
where 
\[
\varLambda_{ij}= \sum_{k,\ell\notin\{i,j\}} M_{ijk}\widetilde{M}_{ij\ell}= J_{ij}^2 B_{ij}, \quad 1\le i<j\le n,
\]
with
\[B_{ij}=\sum_{k,\ell\notin\{i,j\}}J_{ijk}^{(1)}J_{ij\ell}^{(2)}.\] 
We turn to consider each covariance term  on the right-hand side of \eqref{VarVn} 
in the following cases separately:
\begin{enumerate}
    \item[1.] Identical edges:  $(i_1,j_1) = (i_2,j_2)$ (i.e., $i_1=i_2$ and $j_1=j_2$);
    \item[2.] Edges sharing exactly one vertex: $|\{i_1,j_1\} \cap \{i_2,j_2\}| = 1$;
    \item[3.] Disjoint edges: $\{i_1,j_1\} \cap \{i_2,j_2\} = \varnothing$.
\end{enumerate}

For Case 1, we set $(i_1,j_1)=(i_2,j_2)=(i,j)$ for convenience. 
Since $J_{ij}^2$ is independent of all $J_{ijk}^{(1)},J_{ij\ell}^{(2)}$ and $J_{ij}^4=J_{ij}^2$,
we obtain that for any given $i$ and $j$,
\begin{equation}\label{VarvLa}
\mathrm{Var}(\varLambda_{ij})=\E[J_{ij}^4]\E[B_{ij}^2]-(\E[J_{ij}^2])^2(\E[B_{ij}])^2
=2p(1-p)\E[B_{ij}^2]-[2p(1-p)]^2(\E[B_{ij}])^2.    
\end{equation}

By \eqref{EJ12kl}, we have
\begin{equation*}
\E[B_{ij}]=\sum_{k\ne i,j}\E[J_{ijk}^{(1)}J_{ijk}^{(2)}] = (j-i-1)p(1-p)+2(n-j)p(1-p)=O(np).  
\end{equation*}

To estimate $\E[B_{ij}^2]$, expand
\[
B_{ij}^2 = \sum_{k_1,\ell_1\notin\{i,j\}}\sum_{k_2,\ell_2\notin\{i,j\}} 
J_{ijk_1}^{(1)}J_{ij\ell_1}^{(2)}J_{ijk_2}^{(1)}J_{ij\ell_2}^{(2)}.
\]
The expectation of each term is non‑zero only when the indices create sufficient overlap among the four factors 
so that the product does not factor into independent zero‑mean components. 
Consider the cardinality $r=|\{k_1,\ell_1,k_2,\ell_2\}|$.
If $r \ge 3$, at least one factor has mean zero and is independent of the others, so the expectation is zero.
For $r = 2$, non‑zero contributions occur only when the ordered pairs are either identical, 
$(k_1,\ell_1) = (k_2,\ell_2)$, or swapped, $(k_1,\ell_1) = (\ell_2,k_2)$. 
Each such pattern yields $O(n^2)$ quadruples, and each expectation is of order $p^2$, 
contributing $O(n^2p^2)$. 

For $r = 1$, all indices coincide, giving $O(n)$ terms each of order $p$, contributing $O(np)$.
Combining these facts, the dominant contribution comes from the $r=2$ patterns, giving
\[
\E[B_{ij}^2]=O(n^2p^2)+O(np)=O(n^2p^2).
\]

Substituting the estimates of $\E[B_{ij}]$ and $\E[B_{ij}^2]$ into \eqref{VarvLa} gives
\[
\mathrm{Var}(\varLambda_{ij}) = O(n^2p^3).
\]
Therefore, by \eqref{sigman2}, summing over all $N=\binom{n}{2}$ edges gives that the total contribution from Case 1 is
\[
\frac{1}{4\sigma_n^4}\sum_{i<j}\mathrm{Var}(\varLambda_{ij})=O\!\left(\frac{n^2\cdot n^2p^3}{n^6p^4}\right)
=O\!\left(\frac{1}{n^2p}\right)=o\!\left(\frac1n\right),
\]
since $np\to\infty$. Thus Case 1 is asymptotically negligible.

For Case 2, without loss of generality, we let $i_1=i_2=i$.  
All other configurations are symmetric and yield the same estimate.
We now estimate $\operatorname{Cov}(\varLambda_{ij_1},\varLambda_{ij_2})$ where $i,j_1\ne j_2$ are fixed.  
Expanding this covariance gives
\[
\operatorname{Cov}(\varLambda_{ij_1},\varLambda_{ij_2})
= \sum_{k_1,\ell_1\notin\{i,j_1\}}\sum_{k_2,\ell_2\notin\{i,j_2\}}
\operatorname{Cov}\bigl(J_{ij_1}^2 J_{ij_1k_1}^{(1)}J_{ij_1\ell_1}^{(2)},\;
J_{ij_2}^2 J_{ij_2k_2}^{(1)}J_{ij_2\ell_2}^{(2)}\bigr).
\]

Consider the index quadruple $(k_1,\ell_1,k_2,\ell_2)$, each of which is not equal to $i$. 
It is evident that if these indices $k_1,\ell_1,k_2,\ell_2$ are distinct, they contribute nothing to the total covariance
$\operatorname{Cov}(\varLambda_{ij_1},\varLambda_{ij_2})$. Further,
the covariance inside the sum is non-zero only for the following patterns (up to symmetry):
\begin{itemize}
    \item $k_1=k_2$ and $\ell_1=\ell_2$; or $k_1=\ell_2$ and $\ell_1=k_2$, with $k_1,\ell_1$ arbitrary but not in $\{j_1,j_2\}$;
    \item patterns that involve the special vertices $j_1,j_2$, i.e., $k_1=j_2$, $\ell_2=j_1$, $\ell_1=k_2$; $\ell_2=j_1, k_1=\ell_1=k_2$; and their symmetric counterparts;
    \item all four indices equal to a common vertex not in $\{j_1,j_2\}$;
    \item the fixed pattern $k_1=\ell_1=j_2$, $k_2=\ell_2=j_1$.
\end{itemize}
In the first patterns there are $O(n^2)$ choices (two free indices), and each term contributes $O(p^4)$.  
In the second and third patterns there are $O(n)$ choices and each term contributes at most $O(p^3)$. 
The last patterns contribute $O(1)$ terms, each $O(p^2)$.   
Hence the dominant contribution comes from the first patterns, yielding
\[
\bigl|\operatorname{Cov}(\varLambda_{ij_1},\varLambda_{ij_2})\bigr| = O(n^2p^4).
\]
The number of ordered triples $(i,j_1,j_2)$ with distinct vertices is $n\binom{n-1}{2}=O(n^3)$. 
Summing over all such triples yields the total contribution from Case~2 is
\[
\frac{1}{4\sigma_n^4}\sum_{i=1}^n\sum_{\substack{j_1,j_2>i\\ j_1\neq j_2}}
\operatorname{Cov}(\varLambda_{ij_1},\varLambda_{ij_2})
= O\!\left(\frac{n^3\cdot n^2p^4}{n^6p^4}\right) = O\!\left(\frac1n\right).
\]  

As for Case 3, we now consider the covariance $\operatorname{Cov}(\varLambda_{i_1j_1},\varLambda_{i_2j_2})$,
where $\{i_1,j_1\}\cap\{i_2,j_2\}=\varnothing$.
Since in this case, 
$J_{i_1j_1}^2$ and $J_{i_2j_2}^2$ are independent and each is independent of $B_{i_1j_1}$ and $B_{i_2j_2}$,
we have
\begin{align}\label{CovLa3}
 \operatorname{Cov}(\varLambda_{i_1j_1},\varLambda_{i_2j_2})
&= \E\big[J_{i_1j_1}^2\big]\E\big[J_{i_2j_2}^2\big]\operatorname{Cov}(B_{i_1j_1},B_{i_2j_2}) \notag\\
&=[2p(1-p)]^2 \sum_{k_1,\ell_1\notin\{i_1,j_1\}}\sum_{k_2,\ell_2\notin\{i_2,j_2\}}
\operatorname{Cov}\big(J_{i_1j_1k_1}^{(1)}J_{i_1j_1\ell_1}^{(2)},J_{i_2j_2k_2}^{(1)}J_{i_2j_2\ell_2}^{(2)}\big).
\end{align}

Consider the covariance inside the sum. If $k_1$ or $\ell_1$ lies outside $\{i_2,j_2\}$, 
then $J_{i_1j_1k_1}^{(1)}$ or $J_{i_1j_1\ell_1}^{(2)}$ is independent of $J_{i_2j_2k_2}^{(1)}J_{i_2j_2\ell_2}^{(2)}$,
and the corresponding covariance  is equal to zero. Similarly for $k_2,\ell_2$.
A necessary condition for a non‑zero contribution is therefore
\begin{equation}\label{klij}
  \{k_1,\ell_1\}\subseteq\{i_2,j_2\}\quad\text{and}\quad\{k_2,\ell_2\}\subseteq\{i_1,j_1\}.  
\end{equation}
Under \eqref{klij}, each of $k_1,\ell_1$ can only be $i_2$ or $j_2$, 
and each of $k_2,\ell_2$ can only be $i_1$ or $j_1$, giving at most $16$ index quadruples.  
A direct calculation (e.g., for the pattern $k_1=\ell_1=i_2$, $k_2=\ell_2=i_1$) 
shows that each such covariance is of order $p$. Hence,
\[
\sum_{k_1,\ell_1\notin\{i_1,j_1\}}\sum_{k_2,\ell_2\notin\{i_2,j_2\}}
\operatorname{Cov}\big(J_{i_1j_1k_1}^{(1)}J_{i_1j_1\ell_1}^{(2)},J_{i_2j_2k_2}^{(1)}J_{i_2j_2\ell_2}^{(2)}\big)=O(p).
\]
By \eqref{CovLa3}, this implies that 
\[\operatorname{Cov}(\varLambda_{i_1j_1},\varLambda_{i_2j_2})=O(p^3),\]
and thus the total contribution from Case~3 is
\[
\frac{1}{4\sigma_n^4} 
\sum_{i_1<j_1} \sum_{i_2<j_2} 
\mathrm{Cov}\big(\varLambda_{i_1j_1},\varLambda_{i_2j_2}\big)=O\!\left(\frac{n^4\cdot p^3}{n^6p^4}\right)
=O\!\left(\frac{1}{n^2p}\right)=o\!\left(\frac1n\right).
\]
Collecting the above results of the three cases gives that ${\rm Var}(V_n)=O(1/n)$.

It remains to verify the bound for $\operatorname{Var}(V_n^*)$.  
Recall from \eqref{VnstarEF} that 
\[
V_n^* = \frac{1}{\sigma_n^2}\sum_{i<j} J_{ij}^2 
\bigg(\sum_{k\ne i,j} J_{ijk}^{(1)}\bigg) \bigg|\sum_{\ell\ne i,j} J_{ij\ell}^{(2)}\bigg|.
\]
The same decomposition $\varLambda_{ij}=J_{ij}^2B_{ij}$ used for $V_n$ now becomes 
$\varLambda_{ij}^* = J_{ij}^2 \sum_{k,\ell} J_{ijk}^{(1)}\bigl|J_{ij\ell}^{(2)}\bigr|$.  
Since $|J_{ij\ell}^{(2)}|$ satisfies $\E[|J_{ij\ell}^{(2)}|]=O(p)$ and $\E[|J_{ij\ell}^{(2)}|^2]=O(p)$, 
the moment bounds used in the analysis of $V_n$ remain valid with only constant factor changes.  
More concretely, all the combinatorial counting arguments in Cases 1--3 rely only on the independence structure 
and the fact that each factor $J_{ijk}^{(1)}$ or $J_{ij\ell}^{(2)}$ (or its absolute value) 
has mean zero and bounded second moments of order $p$.  
The absolute value does not affect the zero‑mean property of the factors when they appear alone, 
nor does it increase the order of the expectations when multiplied with other factors.  
Hence the same case analysis yields $\operatorname{Var}(V_n^*) = O(1/n)$, and we omit the repetitive details.
\end{proof}

With the above preparations, we now proceed to complete the proof of Theorem \ref{Thm:Fn2}.

\begin{proof}[Proof of (iii) in Theorem \ref{Thm:Fn2}]
First, it follows from Lemma \ref{Shaotheorem} that
\begin{equation}\label{DKHnCN}
 d_{\rm K}(H_n, \mathcal{N})\leq\E\left|1 -\E[V_n\given H_n]\right|
                       +2\E \left| \E[V_n^* \given H_n]\right|.   
\end{equation}
Consider the two mean terms separately. 
Recall from \eqref{EVnVnstar} that $\E[V_n]=1$. 
By the law of total expectation, the random variable $\E[(1-V_n)\given H_n]$ has mean 0.
This yields that
\begin{align*}
\E\left|1 -\E[V_n\given H_n]\right|=\E\left|\E[(1 -V_n)\given H_n]\right|
\le \sqrt{\mathrm{Var}(\E[(1 -V_n)\given H_n])}
\le \sqrt{\mathrm{Var}(1-V_n)}=\sqrt{\mathrm{Var}(V_n)},
\end{align*}
where in the second inequality we used the law of total variance.
Similarly, since $\E[V_n^*]=0$ also by \eqref{EVnVnstar}, we have 
\[
\E \left| \E[V_n^* \given H_n]\right|\le \sqrt{\mathrm{Var}(V_n^*)}.
\]
Substituting these into \eqref{DKHnCN}, together with  Lemma $\ref{upperbound}$,  gives that
\[
 d_{\rm K}(H_n,\mathcal{N})\le\sqrt{\mathrm{Var}(V_n)}+2\sqrt{\mathrm{Var}(V_n^*)}=O\,\Big(\frac1{\sqrt{n}}\Big),
\]
which proves \eqref{HnAN}, by the fact that the convergence in Kolmogorov distance implies the 
convergence in distribution.  
\end{proof}

\begin{remark} {\rm
It is noteworthy that 
the expectation $\E[F_n^2]$ cannot be replaced by its leading term in Theorem \ref{Thm:Fn2},
as Slutsky’s theorem is no longer applicable.
For illustration, consider case (ii) with $p=c(1+\varepsilon_n)/n$, 
where $\varepsilon_n\to 0$ and $\sqrt{n}\varepsilon_n\to \delta$ for some constant $\delta>0$.
Using the expressions \eqref{EFn2} and \eqref{EFnasya} for $\E[F_n^2]$,
a straightforward calculation yields
\[
\E[F_n^2]-\big[(c-\lfloor c\rfloor)^2+\lfloor c\rfloor+2c\big]n=
2c \big(c - \lfloor c \rfloor + 1\big) \delta \sqrt{n}\,(1+o(1)).
\]
This discrepancy is of order $\sqrt{n}$, 
the same as the normalizing constant in Theorem \ref{Thm:Fn2}(ii), 
confirming that centering by the exact expectation is necessary. }
\end{remark}

\begin{remark} {\rm
It is also worth noting that the method used for case~(iii) of Theorem~\ref{Thm:Fn2}, 
which relies on the Berry–Esseen bounds developed by Shao and Zhang
\cite{shao2025104574}, 
is not directly transferable to cases~(i) and~(ii). 
In these regimes, 
$\mathrm{Var}(\widetilde{U}_n)$ is not of smaller order than $\mathrm{Var}(F_n^2)$; 
hence $\widetilde{U}_n$ cannot be neglected, 
and the representation $F_n^2 = -2\widetilde{U}_n - 4(n-2)pE_n + W_n$ 
does not reduce to a simple functional form as tractable as in the sparse cases. 
Moreover, $\widetilde{U}_n$ explicitly involves the Fr\'echet mean graph 
through the coefficients $m_{ij}$. 
Since the Fr\'echet mean set in cases~(i) and~(ii) may contain 
multiple non‑isomorphic graphs, 
these coefficients are not uniquely determined. Consequently, 
constructing the quantities $V_n$ and $V_n^*$ in Lemma~\ref{Shaotheorem} 
would depend on the particular choice of the Fr\'echet mean, 
making a uniform variance estimate prohibitively complicated. 
In contrast, the dependency graph approach adopted for cases~(i) and~(ii) 
circumvents these difficulties by working directly with $F_n^2$ and 
exploiting the sparse dependence structure, 
which becomes increasingly dense in case~(iii) 
and thus necessitates the more refined normal approximation of \cite{shao2025104574}. } 
\end{remark}

\subsection{Proof of Theorem \ref{Thm:Frobenius}}
Based on Theorem~\ref{Thm:Fn2}, we now give a formal proof for Theorem \ref{Thm:Frobenius} using the delta method.

\begin{proof}[Proof of Theorem \ref{Thm:Frobenius}]
It follows by \eqref{Def:a} and \eqref{EFnasya} that
\begin{equation}\label{EFn2asy}B
\E[F_{n}^2]\sim \begin{cases}
 2n^2p(1-p), & \text{if } np(1-p)\to0;\\
 \big[(c-\lfloor c\rfloor)^2+\lfloor c\rfloor+2c\big]n, 
 &\text{if } np(1-p)\to c \text{ for some constant } c>0; \\ 
 3n^2p(1-p), & \text{if } np(1-p)\to\infty.
\end{cases}
\end{equation}
Note that $\E[F_n^2]>0$ in all three regimes, 
and the function $g(x)=\sqrt{x}$ is differentiable at $\E[F_n^2]$ with $g'(x)=1/(2\sqrt{x})>0$.  
Therefore,
applying the delta method to \eqref{Fn2AN1} yields the asymptotic normality of $F_n$ centered at $\sqrt{\E[F_n^2]}$, 
with asymptotic variance
\[
\big[g'(\E[F_n^2])\big]^2\cdot\mathrm{Var}(F_n^2)=\frac{\mathrm{Var}(F_n^2)}{4\E[F_n^2]}.
\]
Using the asymptotic expressions for $\mathrm{Var}(F_n^2)$ from \eqref{VarFn2asy} and for $\E[F_n^2]$ from \eqref{EFn2asy}, a straightforward calculation gives
\begin{equation*}
\frac{\mathrm{Var}(F_n^2)}{4\E[F_n^2]} 
\sim \begin{cases}
1, & \text{if } np(1-p)\to0;\\
\dfrac{c\big[4(c-\lfloor c\rfloor+1)^2+c\big]}
    {2\big[(c-\lfloor c\rfloor)^2+\lfloor c\rfloor+2c\big]},   & \text{if } np(1-p)\to c \text{ for some non-integer constant } c>0; \\ 
 \frac16np(1-p),  & \text{if } np(1-p)\to\infty.
\end{cases}
\end{equation*}
Combining this and applying Slutsky's theorem completes the proof of Theorem~\ref{Thm:Frobenius}.
\end{proof}

\bibliographystyle{plain}
\bibliography{sample}
\end{document}